\documentclass[a4paper,11pt]{amsart}
\usepackage{amsfonts,amssymb,amsmath,amsthm,abstract,color}
\usepackage[mathscr]{euscript}
\usepackage[ps,all,arc,rotate]{xy}
\usepackage[lmargin=1in,rmargin=1in,tmargin=1in,bmargin=1in]{geometry}
\usepackage{fancyhdr}
\usepackage{pb-diagram}
\usepackage{enumitem}
\usepackage{hyperref}
\usepackage[usenames,dvipsnames]{xcolor}
\usepackage{tikz-cd}
\usepackage{extpfeil}
\usepackage{float}
\hypersetup{colorlinks=true,citecolor=NavyBlue,linkcolor=Brown,urlcolor=Orange}

\newtheorem{prop}{Proposition}[section]
\newtheorem{lemma}[prop]{Lemma}
\newtheorem{thm}[prop]{Theorem}
\newtheorem{cor}[prop]{Corollary}

\theoremstyle{definition}
\newtheorem{defn}[prop]{Definition}

\newtheorem{rmk}[prop]{Remark}

\renewcommand{\tilde}{\widetilde}     
\newcommand{\ra}{\rightarrow}

\DeclareMathOperator{\Jac}{Jac}

\DeclareMathOperator{\Pic}{Pic}

\DeclareMathOperator{\Sym}{Sym}

\DeclareMathOperator{\Hom}{Hom}

\DeclareMathOperator{\Aut}{Aut}

\def\DM{\mathrm{DM}}

\def\CH{\mathrm{CH}} 
\def\CHM{\mathrm{CHM}}

\DeclareMathOperator{\eff}{eff}

\DeclareMathOperator{\SmProj}{SmProj}
\DeclareMathOperator{\codim}{codim}

\def\num{{\mathrm{num}}}

\def\cL{\mathcal L}
\def\cM{\mathcal M}\def\cN{\mathcal N}\def\cO{\mathcal O}
\def\cP{\mathcal P}

\def\CC{\mathbb C}

\def\GG{\mathbb G}
\def\LL{\mathbb L}
\def\NN{\mathbb N}
\def\PP{\mathbb P}\def\QQ{\mathbb Q}

\def\ZZ{\mathbb Z}

\def\fh{\mathfrak h}

\def\k{\mathbf k}

\makeatletter
\newcommand{\flip}[2][]{\ext@arrow 3359\leftrightarrowfill@@{#1}{#2}}
\def\rightarrowfill@@{\arrowfill@@\relax\relbar\rightarrow}
\def\leftarrowfill@@{\arrowfill@@\leftarrow\relbar\relax}
\def\leftrightarrowfill@@{\arrowfill@@\leftarrow\relbar\rightarrow}
\def\arrowfill@@#1#2#3#4{%
  $\m@th\thickmuskip0mu\medmuskip\thickmuskip\thinmuskip\thickmuskip
   \relax#4#1
   \xleaders\hbox{$#4#2$}\hfill
   #3$%
}
\makeatother

\title[Motives of moduli spaces of rank 3 bundles]{Motives of moduli spaces of rank $3$ vector bundles and Higgs bundles on a curve}

\author{Lie Fu, Victoria Hoskins and Simon Pepin Lehalleur}

\thanks{L. F. is supported by the Radboud Excellence Initiative programme and the Agence Nationale de la Recherche (ANR), under project numbers ANR-16-CE40-0011 and ANR-20-CE40-0023.. S. P. L. is supported by The Netherlands Organisation for Scientific Research (NWO), under project number 613.001.752.}

\begin{document}

\maketitle

\begin{abstract}
We prove formulas for the rational Chow motives of moduli spaces of semistable vector bundles and Higgs bundles of rank $3$ and coprime degree on a smooth projective curve. Our approach involves identifying criteria to lift identities in (a completion of) the Grothendieck group of effective Chow motives to isomorphisms in the category of Chow motives. For the Higgs moduli space, we use motivic Bia{\l}ynicki-Birula decompositions associated with a scaling action, together with the variation of stability and wall-crossing for moduli spaces of rank $2$ pairs, which occur in the fixed locus of this action.
\end{abstract}

\tableofcontents

\section{Introduction}

Let $C$ be a smooth projective geometrically connected curve of genus $g\geq 1$ over a field $\k$. We assume that $C$ admits a degree $1$ line bundle. Let $\cN = \cN_C(n,d)$ (resp.~$\cM = \cM_C(n,d)$) denote the moduli space of semistable vector bundles (resp.~Higgs bundles) of rank $n$ and degree $d$ on $C$.  Throughout this paper we assume that $n$ and $d$ are coprime, so that semistability and stability coincide. The variety $\cN$ is smooth projective of dimension $n^2(g-1) +1$ and the variety $\cM$ is  smooth quasi-projective of $\dim \cM = 2 \dim \cN$.


The cohomology of both $\cN$ and $\cM$ have been extensively studied; different approaches to describe their various cohomological invariants should be both unified and refined by working with \textit{motivic invariants}, which encode finer invariants, like Hodge structures on cohomology groups and also algebro-geometric invariants such as Chow groups. Let us explicitly mention some motivic descriptions of these moduli spaces. The motivic Poincar\'{e} polynomial of the vector bundle moduli space $\cN$ was computed by del Ba\~no \cite{dB_motive_moduli_vb} using the geometric techniques of \cite{BGL}; the ideas in \cite{BGL} were also used to give formulas for the stack of vector bundles on $C$ in the Grothendieck ring of varieties \cite{BD} and in Voevodsky's triangulated category of motives over $\k$ with rational coefficients \cite{HPL_formula}. 

An algorithm for computing the class of the Higgs moduli space $\cM$ in the Grothendieck ring of varieties was described by Garc\'{i}a-Prada, Heinloth and Schmitt \cite{GPHS} using the Bia{\l}ynicki-Birula decomposition associated to the natural scaling action on $\cM$ considered by Hitchin \cite{Hitchin} in rank 2, Gothen \cite{Gothen} in rank 3, and Simpson \cite{Simpson}, together with variation of stability for chains of vector bundle homomorphisms. This was upgraded in \cite{HPL_Higgs} to a motivic argument in Voevodsky's triangulated category of motives with rational coefficients and, by \cite[Corollary 6.9]{HPL_Higgs}, the motive of $\cM$ is pure and lies in the tensor subcategory generated by the motive of $C$. 

In a recent paper \cite{FHPL-rank2}, we thoroughly studied the rank 2 case and gave formulas for the rational Chow motives of $\cN_C(2,d)$ and $\cM_C(2,d)$, as well as moduli spaces of parabolic bundles and parabolic Higgs bundles by using explicit descriptions of variation of stability as flips. 

In this paper, we proceed to rank 3 and give formulas for the rational Chow motives of $\cN=\cN_C(3,d)$ and $\cM=\cM_C(3,d)$ for $d$ coprime to $3$. We expect that similar techniques to \cite{FHPL-rank2} can be used to give formulas for the motives of moduli spaces of parabolic vector bundles.

\subsection{The motive of the vector bundle moduli space}

For an integer $d$ coprime to 3 and $\cL\in \Pic^d(C)$, we let $\cN_{\cL} = \cN_{C,\cL}(3,d)$ denote the moduli space of semistable vector bundles with determinant isomorphic to $\cL$. The rational Chow motive $\fh(\cN_{\cL})$ is abelian by \cite[Proposition 4.1]{FHPL-rank2}, and $\fh(\cN(3,d))\simeq \fh(\cN_{\cL}(3, d))\otimes \fh(\Jac(C))$ by \cite[Theorem 1.1]{FHPL-rank2}. Hence it suffices to give a formula for the motive $\fh(\cN_{\cL}(3, d))$.

\begin{thm}
	\label{thm:main thm N}
Assume that $C$ has a degree $1$ line bundle and $d$ is coprime to 3. For any $\cL\in \Pic^d(C)$, the rational Chow motive of $ \cN_{C,\cL}(3,d)$ is
\[ \fh(\cN_{C,\cL}(3,d))  \simeq \fh(C^{(g-1)}\times C^{(g-1)})(3g-3) \oplus \bigoplus_{\substack{k_1 + k_2 < 2g-2 \text{ or }\\ k_1+k_2=2g-2 \text{ and } k_1<g-1}} \fh(C^{(k_1)} \times C^{(k_2)}) \otimes L_{k_1,k_2},\]
where $L_{k_1,k_2}$ are sums of Tate twists given by  $L_{k_1,k_2}= \QQ(k_1 + 2k_2) \oplus \QQ(8g-8 -2k_1 -3k_2)$.
\end{thm}

This theorem upgrades a recent computation of the motivic Poincar\'{e} polynomial of $\cN_{C,\cL}(3,1)$ in the completion $\widehat{K_0}(\CHM^{\eff}(\k,\QQ))$ of the Grothendieck group of effective Chow motives over $\k$ along the ideal generated by the Lefschetz motive $\LL=\QQ(1)$ due to Gomez and Lee \cite{GL} to an isomorphism of Chow motives. In particular, Theorem \ref{thm:main thm N} gives information about the Chow groups of $\cN_{C,\cL}(3,d)$ and $\cN_C(3,d)$ (see \S \ref{subsec:Chowgroups} for some examples), which is not captured by the previous results on the motivic Poincar\'{e} polynomial.

In general, it is not known if either the map from isomorphism classes in $\CHM^{\eff}(\k,\QQ)$ to $K_{0}(\CHM^{\eff}(\k,\QQ))$ or the natural ring homomorphism $K_{0}(\CHM^{\eff}(\k,\QQ))\to \widehat{K_0}(\CHM^{\eff}(\k,\QQ))$ are injective. However, in the case of Kimura finite dimensional motives \cite{Kimura-FiniteDimension} (for example, abelian motives), we show that it is possible to lift identities in $\widehat{K_0}(\CHM^{\eff}(\k,\QQ))$ to isomorphisms in $\CHM(\k, \QQ)$ in $\S$\ref{sec lifting identities}. 

This strategy would theoretically enable one to obtain formulas for the Chow motives of $\cN_{C,\cL}(n,d)$ and $\cN_{C}(n,d)$ in higher ranks; however, one can only apply Corollary \ref{cor lifting effective identities} to lift identities in $\widehat{K_0}(\CHM^{\eff}(\k,\QQ))$ with \emph{positive coefficients} to $\CHM(\k,\QQ)$. Since the motivic Poincar\'{e} polynomial of $\cN_{C,\cL}(n,d)$ is computed in \cite{dB_motive_moduli_vb} using a Harder--Narasimhan recursion which involves introducing negative signs, it remains to write these as positive identities in order to obtain corresponding isomorphisms of Chow motives. This appears to be a difficult combinatorial problem in general, which was solved by Gomez and Lee in \cite{GL} in rank 3.

\subsection{The motive of the Higgs moduli space}

To study the Chow motive of the Higgs moduli space $\cM =\cM_C(3,d)$ we use the motivic Bia{\l}ynicki-Birula decomposition associated to the $\GG_m$-action on the Higgs moduli space given by scaling the Higgs field. In rank $3$, this idea was used by Gothen \cite{Gothen} to compute the Poincar\'{e} polynomial of $\cM$ (and of the moduli space of Higgs bundles with fixed determinant). The motive of $\cM$ is expressed in terms of Tate twists of the motives of the fixed components, which come in four types: one fixed components is $\cN$ where the Higgs field is zero, then there are components of Type (1,1,1) where the underlying bundle decomposes as a sum of line bundles $E = L_1 \oplus L_2 \oplus L_3$ each of which is sent to the next via the Higgs field, and finally there are components of Type (1,2) and Type (2,1) where $E$ decomposes as $E = L \oplus F$ and $E = F \oplus L$ respectively for a line bundle $L$ and the Higgs field maps the first factor to the second. 

Gothen showed that the components of Type (1,2) and (2,1) are related to moduli spaces of pairs consisting of a rank $2$ vector bundle and a non-zero section which are semistable with respect to an appropriate stability parameter. The variation of stability for these  moduli spaces  of rank $2$ pairs was studied by Thaddeus \cite{Thaddeus_pairs}, where he explicitly described the birational wall-crossing transformations as standard flips (or flops). From this description, it is straightforward to compute the motivic Poincar\'{e} polynomial of moduli spaces of rank $2$ pairs which are semistable with respect to a generic stability parameter (i.e. where semistability and stability coincide). However, in general, these wall-crossing formulas involve negative signs. At this point, we again employ the strategy of $\S$\ref{sec lifting identities}: we express the motivic Poincar\'{e} polynomial of the  moduli spaces of rank $2$ pairs that we are interested in as a \textit{positive} combination of motives, so that it can be lifted to an isomorphism in $\CHM(\k,\QQ)$. In fact, we give two different formulas for the rational Chow motives of these moduli spaces of rank $2$ degree $e$ pairs $\cP_e^i$ (here $i$ indexes the chamber in which we take our stability parameter). Let us state the more compact geometric formulation here, which expresses the Chow motives of pair moduli spaces in terms of motives of symmetric powers of $C$ and $\Jac(C)$; we refer the reader to Corollary \ref{cor Chow motives pairs} for the alternative formula in terms of $\fh^1(C)$. For a polynomial with positive integral coefficients  $Q(T)=\sum_k a_k T^k\in \NN[T]$, we define the Tate motive $Q(\QQ(1)):=\bigoplus_k \QQ(k)^{\oplus a_k}\in \CHM(\k, \QQ)$. 

\begin{thm}\label{main thm P}
Assume that $2i<e\leq 4g-5$. Then the rational Chow motive of the moduli space $\cP_e^i$ of rank $2$ degree $e$ pairs with stability given by the $i$th chamber is computed as follows:
	\begin{enumerate}[label=\emph{(\roman*)}, leftmargin=0.7cm]
		\item If $3i< e+g$, then \[ \fh(\cP^i_{e}) \simeq \bigoplus_{k=0}^i \fh(\Jac(C)\times C^{(k)}\times \PP^{e+g-3k-2})(k) .\]
		\item If $3i\geq e+g$, then 
			\begin{align*}
		\fh(\cP_e^i)\simeq & \: \fh(\Jac(C)\times \PP^{e-2g+1}\times C^{(g-1)})(g-1)\oplus\bigoplus_{k=0}^{2g-3-i}\fh(\Jac(C)\times C^{(k)}\times \PP^{e+g-3k-2})(k)\\
		&\oplus \bigoplus_{k=2g-2-i}^{g-2}\fh(\Jac(C)\times C^{(k)}\times \PP^{e-2g+1})\otimes(\QQ(3g-3-2k)+\QQ(k))\\
		&\oplus\fh(\Jac(C)^2)\otimes Q_{i,e,g}(\QQ(1)),
		\end{align*}
		where $Q_{i,e,g}(T)\in \NN[T]$ is defined as follows
		\begin{equation*}
		Q_{i, e, g}(T):= \frac{(T^g-T^{e+g-1-2i})(1-T^{i-g+1})(1-T^{i-g+2})}{(1-T)^2(1-T^2)}.
		\end{equation*}
	\end{enumerate}
\end{thm}

Finally, here is the formula we then obtain for the motive of moduli spaces of  rank 3 Higgs bundles. Recall that the Voevodsky motive of this quasi-projective moduli space turns out to be pure (\cite[Corollary 6.9]{HPL_Higgs}), thus it makes sense to speak of its Chow motive.

\begin{thm}
	\label{main thm H}
Assume that $C$ has a degree $1$ line bundle and $d$ is coprime to 3. The rational Chow motive of the Higgs moduli space $\cM_C(3, d)$ is given by the following expression.
		\begin{align*}
	\fh(\cM_C(3,d))\simeq & \: \fh(\cN_C(3,d)) \oplus \bigoplus_{l=0}^{g-2}\fh(\Jac(C))\otimes \left[\fh(\cP_{4g-3l-6}^{2g-2l-4})(2g+3l) \oplus \fh(\cP_{4g-3l-5}^{2g-2l-3})(2g+3l-1) \right] \\
	&\oplus \bigoplus_{\substack{(m_1,m_2) \in \NN^2\\2m_1+m_2<6g-6\\2m_2+m_1<6g-6\\m_2\equiv m_1+1 \operatorname{mod} 3}}\fh(\Jac(C))\otimes\fh(C^{(m_1)}\times C^{(m_2)})(8g-8-m_1-m_2).
	\end{align*}
where $\fh(\cN_C(3,d)) \simeq \fh(\cN_{C,\cL}(3,d)) \otimes \fh(\Jac(C))$ and the Chow motives of the moduli space $\cN_{C,\cL}(3,d)$ of vector bundles with fixed determinant and the pair moduli spaces $\cP^i_e$ appearing here are calculated by Theorem \ref{thm:main thm N} and Theorem \ref{main thm P} (or Corollary \ref{cor Chow motives pairs}) respectively.
\end{thm}

By plugging in the formulas from Theorem \ref{thm:main thm N} and Theorem \ref{main thm P}, we see that the motive of $\cM_C(3,d)$ is expressed in terms of motives of $\Jac(C)$ and symmetric powers of $C$.

Theorem \ref{main thm H} can be viewed as a lifting of the formula of Garc\'ia-Prada--Heinloth--Schmitt \cite[\S 8]{GPHS} for the class of $\cM$ in a certain completion of the Grothendieck ring of varieties. Actually, it takes some straightforward but tedious computation to see that our isomorphism indeed recovers the formula in \cite{GPHS}. 

We note that the relationship between the Higgs moduli space $\cM$ and the Higgs moduli space $\cM_{\cL}$ with fixed determinant is not as simple as the case for $\cN$ and $\cN_{\cL}$ given by \cite[Theorem 1.1]{FHPL-rank2}; this was already observed on the level of cohomology in rank $n =2$ by Hitchin \cite{Hitchin} and in rank 3 by Gothen \cite{Gothen}. In fact, in Proposition \ref{prop:MotiveHiggsFixDetNotGeneratedByC}, we show that for a general smooth projective complex curve $C$, the rational Chow motive of $\cM_{\cL,C}(n,d)$ for any $n$ and any $d$ coprime to $n$ is \emph{not} contained in the tensor subcategory generated by the motive of $C$.

Finally, we give some explicit formulas in low genus and applications to Chow groups in $\S$\ref{sec ex app}.

\subsection*{Notation} 
Throughout $C$ denotes a smooth projective geometrically connected curve over a field $\k$ which we assume admits a degree $1$ line bundle. We let $C^{(m)}$ denote the $m$-fold symmetric power of $C$ and let $\Jac(C)$ denote the Jacobian of $C$. 

We write $\CHM(\k,\QQ)$ for the category of Chow motives over $\k$ with coefficients in $\QQ$ and we follow a homological convention for morphisms. With this convention, we have a covariant functor $\fh:\SmProj_{\k}\to \CHM(\k,\QQ)$ associating to a smooth projective $\k$-variety $X$ its rational Chow motive $\fh(X)$. In particular, for a smooth $\k$-variety $X$ whose motive is pure, the rational Chow groups of $X$ can be computed as homomorphism groups in $\CHM(\k,\QQ)$ as follows
\begin{equation}\label{eq Chow gps as Hom gps}
\CH^i(X)_{\QQ}=\Hom_{\CHM(\k,\QQ)}(\fh(X), \QQ(i)).
\end{equation} 

Since the Higgs moduli space $\cM$ is only quasi-projective, the natural way to associate a motive to $\cM$ is via the triangulated category $\DM(\k,\QQ)$ of Voevodsky motives over $\k$ with rational coefficients. However, by \cite[Corollary 6.9]{HPL_Higgs}, this motive is pure and thus we can view it as a Chow motive by identifying $\CHM(\k,\QQ)$ with a full subcategory of $\DM(\k,\QQ)$ via the fundamental embedding theorem of Voevodsky \cite{VoevodskyBookChapter} (see \cite[$\S$2.2]{FHPL-rank2}).

\section{Lifting identities from the Grothendieck group to Chow motives}\label{sec lifting identities}

\begin{defn}
The \textit{motivic Poincar\'{e} polynomial} $\chi(X)$ of a smooth projective $\k$-variety $X$ is the image of the rational Chow motive $\fh(X)$ in the completion $\widehat{K_0}(\CHM^{\eff}(\k,\QQ))$ of the Grothendieck ring of effective rational Chow motives over $\k$ along the ideal generated by the Lefschetz motive $\LL=[\QQ(1)]$.
\end{defn}

This notion was introduced in \cite{dB_motive_moduli_vb}, because power series in $\LL$ occur naturally and unavoidably in his computations. In general, it is not known if either the natural  ring homomorphism $K_{0}(\CHM^{\eff}(\k,\QQ))\to \widehat{K_0}(\CHM^{\eff}(\k,\QQ))$ or the map from isomorphism classes in $\CHM^{\eff}(\k,\QQ)$ to $K_{0}(\CHM^{\eff}(\k,\QQ))$ are injective. However, with Kimura's assumption on finite dimensionality \cite{Kimura-FiniteDimension}, we have the following result (already appeared in \cite[Lemme 13.2.1.1]{AndreBook} and exploited in \cite[Theorem 4.3]{FHPL-rank2}).

\begin{prop}
	\label{prop:UpgradeK0toCHM}
Given effective Chow motives $N, M \in \CHM^{\eff}(\k,\QQ)$ which are Kimura finite dimensional and whose classes in $\widehat{K_0}(\CHM^{\eff}(\k,\QQ))$ coincide, there is an isomorphism $N \simeq M$ in $\CHM^{\eff}(\k,\QQ)$.
\end{prop}
\begin{proof}
The equality in $\widehat{K_0}(\CHM^{\eff}(\k,\QQ))$ also determines a corresponding isomorphism in the category $M^{\eff}_{\num}(\k,\QQ)$ of effective numerical motives (because $M^{\eff}_{\num}(\k,\QQ)$ is semisimple \cite{Jannsen}, see the proof of \cite[Theorem 4.3]{FHPL-rank2}). In particular, the equality holds in $M_{\num}(\k,\QQ)$. Since the restriction of the functor from Chow motives to numerical motives to the subcategory of Kimura finite dimensional Chow motives is full and conservative (see \cite{AndreBourbaki}), we deduce that the isomorphism also holds in $\CHM(\k,\QQ)$. 
\end{proof}

We note that the above isomorphism is not explicit. More precisely, we can lift \emph{positive} identities in $\widehat{K_0}(\CHM^{\eff}(\k,\QQ))$ to isomorphisms in $\CHM(\k, \QQ)$, when the terms appearing are Kimura finite dimensional, as in the following corollary.

\begin{cor}\label{cor lifting effective identities}
Let $X$ be a smooth projective variety, whose Chow motive $\fh(X)$ is Kimura finite dimensional (for example, $\fh(X)$ is an abelian motive). Given any effective identity in $\widehat{K_0}(\CHM^{\eff}(\k,\QQ))$ expressing the motivic Poincar\'{e} polynomial $\chi(X)$ as a polynomial in $\LL$
\[ \chi(X) = \sum_{i = 0}^N \chi(X_i) \LL^{i} \] 
whose coefficients are motivic Poincar\'{e} polynomials of smooth projective varieties $X_i$ with $\fh(X_i)$ being Kimura finite dimensional, there is a corresponding isomorphism in $\CHM(\k,\QQ)$; that is,
\[ \fh(X) \simeq \bigoplus_{i=0}^N \fh(X_i)(i). \]
\end{cor}

This corollary also holds when $X$ and $X_i$ are varieties (not necessarily smooth or projective) whose motives are pure.

\section{Motives of moduli spaces of vector bundles of rank 3}

For two coprime positive integers $n$ and $d$, let $\cN = \cN_C(n,d)$ denote the moduli space of semistable vector bundles of rank $n$ and degree $d$ on $C$. If $\cL$ is a degree $d$ line bundle on $C$, we can also consider the moduli space $\cN_{\cL} = \cN_{C,\cL}(n,d)$ of rank $n$ semistable vector bundles with determinant $\cL\in \Pic^d(C)$. In \cite{FHPL-rank2}, we computed the rational motives of $\cN$ and $\cN_{\cL}$ for $n = 2$. In this section, we will consider the case of rank $n = 3$.

Our starting point is a formula for the motivic Poincar\'{e} polynomial $\chi(\cN_{\cL})$ recently established by Gomez and Lee \cite[Theorem 1.3]{GL}, which goes back to the work of del Ba\~no \cite{dB_motive_moduli_vb} on the motive of $\cN_{\cL}$.


\begin{thm}[Gomez--Lee] 
	Let $\cL$ be a degree 1 line bundle on $C$.
	The motivic Poincar\'{e} polynomial of $\cN_{\cL}(3,1)$ is as follows:
	\[ \chi(\cN_{\cL}(3,1))  = \chi(C^{(g-1)} \times C^{(g-1)}) \LL^{3g-3} + \sum_{\substack{k_1 + k_2 < 2g-2 \text{ or }\\ k_1+k_2=2g-2 \text{ and } k_1<g-1}} \chi(C^{(k_1)} \times C^{(k_2)}) \LL_{k_1,k_2}\]
	where $\LL_{k_1,k_2}= \LL^{k_1 + 2k_2} + \LL^{8g-8 -2k_1 -3k_2}$.
\end{thm}

\begin{rmk}[Independency on the degree $d$ in rank 3]
	\label{rmk:IndependentOfDegree}
	Note that the operation of taking dual bundles $E\mapsto E^\vee$ preserves stability, hence gives rise to an isomorphism of moduli spaces $\cN_{\cL}(n, d)\simeq \cN_{\cL^\vee}(n, -d)$ and $\cN(n, d)\simeq \cN(n, -d)$. Now specializing to rank $n=3$ and  assuming the existence of a degree 1 line bundle on $C$, we have $\cN(3, 1)\simeq\cN(3, -1)$
	and $\cN(3, d)\simeq \cN(3, d+3)$.
	 Hence when $\Pic^1(C)\neq \emptyset$,  the isomorphism class of $\cN(3, d)$ is independent of $d$, provided that $(3, d)=1$. 
\end{rmk}

We use a similar trick to \cite[Theorem 4.3]{FHPL-rank2} as described in $\S$\ref{sec lifting identities} above to upgrade Gomez--Lee's identity in $\widehat{K_0}(\CHM^{\eff}(\k,\QQ))$ to an isomorphism in the category $\CHM(\k,\QQ)$.

\begin{thm}
	\label{thm:Motive VB rank 3}
	Let $C$ be a smooth projective curve defined over $\k$ admitting a degree 1 line bundle. 
	For any $d\in \ZZ$ which is coprime to 3  and  for any $\cL\in \Pic^d(C)$, the rational Chow motive of $\cN_{C,\cL}(3,d)$ is
\[ \fh(\cN_{C,\cL}(3,d))  \simeq \fh(C^{(g-1)}\times C^{(g-1)})(3g-3) \oplus \bigoplus_{\substack{k_1 + k_2 < 2g-2 \text{ or }\\ k_1+k_2=2g-2 \text{ and } k_1<g-1}} \fh(C^{(k_1)} \times C^{(k_2)}) \otimes L_{k_1,k_2},\]
where $L_{k_1,k_2}$ are sums of Tate twists given by  $L_{k_1,k_2}= \QQ(k_1 + 2k_2) \oplus \QQ(8g-8 -2k_1 -3k_2)$.\\
 The rational Chow motive of $\cN=\cN_C(3, d)$ is 
 \[\fh(\cN(3, d))\simeq \fh(\cN_{\cL}(3, d))\otimes \fh(\Jac(C)).\]
\end{thm}
\begin{proof}
In view of Remark \ref{rmk:IndependentOfDegree}, we can assume without loss of generality that $d=1$.

Since the formula of Gomez and Lee in $\widehat{K_0}(\CHM^{\eff}(\k,\QQ))$ is an equality between virtual motives with positive coefficients and the Chow motive $\fh(\cN_{\cL}(3,d))$ is abelian by \cite[Theorem 4.1]{FHPL-rank2} (and thus Kimura finite dimensional \cite[Th\'eor\`eme 2.8]{AndreBourbaki}), we can apply Corollary \ref{cor lifting effective identities} to deduce the claimed isomorphism in $\CHM(\k,\QQ)$. 

The final isomorphism follows from \cite[Theorem 1.1]{FHPL-rank2}.
\end{proof}

\section{Motives of moduli spaces of rank 2 pairs}

\subsection{Moduli spaces of rank two pairs}

Moduli of pairs $(V,\phi)$ consisting of a vector bundle $V$ on $C$ and a non-zero section $\phi \in H^0(C,V)$  have been studied by Bradlow \cite{Bradlow} and Thaddeus \cite{Thaddeus_pairs}, who gave a GIT construction of pair moduli spaces depending on a stability parameter $\sigma \in \QQ_{>0}$. In this section, we focus on moduli spaces of rank $2$ pairs, as in our later application to Higgs bundles, rank $2$ pairs are related to rank $3$ Higgs bundles via a Bia{\l}ynicki-Birula decomposition studied by Gothen \cite{Gothen} (see $\S$\ref{sec Type 12} and $\S$\ref{sec Type 21} below). 

\begin{defn}
A rank 2 pair $(V,\phi)$ is $\sigma$-\textit{semistable} if for all line subbundles $M \subset V$, we have 
\[ \mu(M) \leq \mu(V) - \epsilon(M,\phi) \sigma \quad \text{where } \quad \epsilon(M,\phi) := \left\{ \begin{array}{cl} 1 & \text{if } \phi \in H^0(C,M) \\ -1 & \mathrm{else.}
\end{array} \right. \] 
If this equality is strict for all $M$, we say that $(V,\phi)$ is $\sigma$-stable.
\end{defn}

Fix a stability parameter $\sigma$ and a degree $e>0$ (resp. a degree $e$ line bundle $\cL$ on $C$); then there is a projective moduli space $\cP^{\sigma-ss} = \cP^{\sigma-ss}_C(2,e)$ (resp. $\cP^{\sigma-ss}_{\cL}$) of $\sigma$-semistable pairs on $C$ of rank $2$ and degree $e$ (resp. with determinant isomorphic to $\cL$) constructed as a GIT quotient \cite{Thaddeus_pairs}.  For generic $\sigma$ (where semistability and stability coincide), the pair moduli spaces are smooth and have dimension
\[ \dim (\cP^{\sigma-ss}_C(2,e)) = e + 2g -2.\]

\subsection{Motives of moduli spaces of pairs via variation of stability and flips}
The motives of moduli spaces of pairs will naturally appear in $\S$\ref{sec Type 12} and $\S$\ref{sec Type 21} below.  To study their motives, we use Thaddeus' description \cite{Thaddeus_pairs} of the birational transformations between these moduli spaces as the stability parameter $\sigma$ varies, as explicit standard flips. This immediately gives rise to formulas for the motivic Poincar\'{e} polynomials of moduli spaces of pairs in the Grothendieck ring of Chow motives and then we apply Proposition \ref{prop:UpgradeK0toCHM} by finding a positive expression for the motivic Poincar\'{e} polynomials of moduli spaces of pairs.

Let us first recall the notion of standard flips/flops.

\begin{defn}[Standard flip]
	\label{def flip}
	Let $S$ be a smooth projective variety and $\phi \colon X \dashrightarrow X'$ be a birational transformation between smooth projective varieties. We say that $\phi$ is a \textit{standard flip of type $(m,l)$ with centre $S$} if there are closed smooth subvarieties $Z \hookrightarrow X$ and $Z' \hookrightarrow X'$ which are projective bundles over $S$ of relative dimensions $m$ and $l$ respectively such that the blow-up $\tilde{X}$ of  $X$ along $Z$ coincides with the blow-up of $X'$ along $Z'$ with common exceptional divisor $E$. This is summarised by the following diagram
	\[ \xymatrixcolsep{1cm} \xymatrixrowsep{0.5cm} \xymatrix{
		& & E  \ar[lldd] \ar@{^{(}->}[d] \ar[rrdd] & & \\
		& & \tilde{X} \ar[ld] \ar[rd] & & \\
		Z \: \ar[rrdd]_{\PP^m} \ar@{^{(}->}[r] & X \ar@{-->}[rr]^{\text{type } (m,l)}_{\text{centre } S} & & X' & \: Z' \ar@{_{(}->}[l] \ar[lldd]^{\PP^l} \\ & & & & \\
		& & S & &
	}\]  
	where the top two squares are blow-up squares and the outer square is cartesian. When $m = l$, we call this a standard flop.
\end{defn}

Thaddeus studied variation of stability for rank $2$ degree $e$ pairs $(V,\phi)$ in \cite{Thaddeus_pairs}. The space of stability parameters $\QQ_{>0}$ admits a wall and chamber decomposition given by considering how the notion of (semi)stability changes as $\sigma$ varies. The moduli space $\cP^\sigma(2, e)$ is non-empty if and only if $\sigma \leq e/2$ by \cite[(1.3)]{Thaddeus_pairs}, and so we can restrict our attention to the interval $(0,e/2]$. There is a finite set $W \subset (0,e/2]$ of walls; that is, critical values of $\sigma$ for which semistability and stability do not coincide (i.e. there is a line subbundle of degree $e'$ with $e' = e/2  \pm \sigma$). The walls are given by $\sigma_0> \dots >\sigma_{m}$ where $m = \lfloor (e-1)/2 \rfloor$ and $\sigma_i :=e/2 - i$ by purely numerical considerations. The connected components of $(0,e/2] \setminus W$ are called chambers which we also label from right to left as $C_i = (\sigma_{i+1},\sigma_i)$, where for notational convenience, we set $\sigma_{m+1}:=0$ (which is not a stability parameter). In each chamber $C_i$, semistability and stability coincide and the corresponding moduli space of pairs is smooth and only depends on the chamber; thus we write  $\cP^i_e:=\cP^{\sigma-ss}(2,e)$ (resp. $\cP^i_{\cL} :=\cP_{\cL}^{\sigma-ss}(2,e)$) for the moduli space of pairs (resp. with determinant $\cL$) which are stable with respect to $\sigma \in C_i$. 

\begin{thm}[Thaddeus \cite{Thaddeus_pairs}]\label{thm pairs VGIT}
	For degree $e \geq 3$, let  $m = \lfloor (e-1)/2 \rfloor$ and let  $\cP^0_{\cL}, \dots , \cP^m_{\cL}$ be the moduli spaces of stable pairs of rank $2$ vector bundles with determinant a degree $e$ line bundle $\cL$ and for stability parameters appearing in each of the $C_0,\dots,C_m$ chambers introduced above.
	\begin{enumerate}[label=\emph{(\roman*)}]
		\item \label{pairs large stab} The extremal moduli space $\cP^0_{\cL}$ is the moduli space of non-split extensions of $\cL$ by $\cO_C$  and consequently we have $\cP^0_{\cL}\cong \PP(H^1(C,\cL^{-1})) \cong \PP^{e+g-2}$.
		\item \label{pairs small stab} The opposite extremal moduli space $\cP^m_{\cL}$ admits a natural map $\pi : \cP^m_{\cL} \ra \cN_{\cL}=\cN_{\cL}(2,e)$ given by forgetting the section. If $e > 2g-2$, then $\pi$ is surjective with fibre $\PP(H^0(C,E))$ over $E \in \cN_{\cL}$. If $e > 4g-4$, then $\pi$ is a $\PP^{e-2g+1}$-bundle.
		\item \label{pairs flips} There is a standard flip $\cP^{i-1}_{\cL} \dashrightarrow \cP^i_{\cL}$ of type $(i-1,e+g-2(i+1))$ with centre $C^{(i)}$.
	\end{enumerate}
\end{thm}

If one does not fix the determinant $\cL$, then there is a similar picture where (i) $\cP^0_e \ra \Pic^e(C)$ is a $\PP^{e+g-2}$-fibration, (ii) there is a forgetful map $\pi: \cP^m_e \ra \cN$ which is a $\PP^{e-2g+1}$-fibration for $e > 4g-4$ and (iii) one replaces the centres of the flips $C^{(i)}$ with $\Pic^{e-i}(C) \times C^{(i)}$ (see \cite[$\S$8]{Thaddeus_VGIT}). More precisely, the relationship between these moduli spaces of stable pairs is illustrated by the following diagram
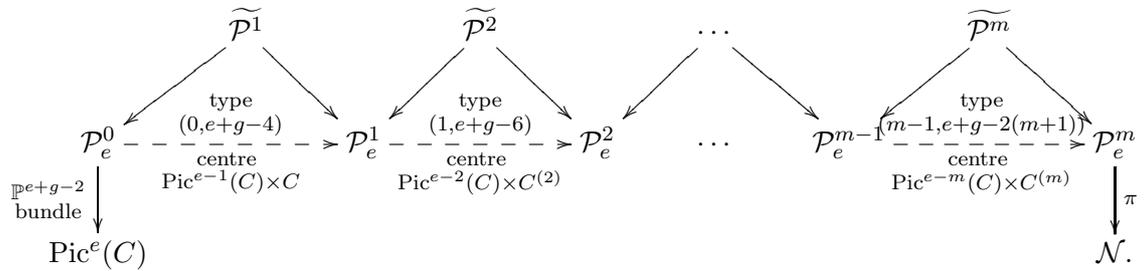
\[   \xymatrix{ & \tilde{\cP^1} \ar[dl] \ar[dr] & & \tilde{\cP^2} \ar[dl] \ar[dr] & & \cdots \ar[dl] \ar[dr] & & \tilde{\cP^m} \ar[dl] \ar[dr] \\
	\cP^0_{e} \ar[d]_{\begin{smallmatrix} \PP^{e+g-2} \\ \text{bundle} \end{smallmatrix}} \ar@{-->}[rr]^{\begin{smallmatrix} \text{type} \\ (0,e+g -4) \end{smallmatrix}}_{\begin{smallmatrix} \text{centre }  \\ \Pic^{e-1}(C) \times C  \end{smallmatrix}} &  & \cP^1_{e}  \ar@{-->}[rr]^{\begin{smallmatrix} \text{type} \\ (1,e+g -6) \end{smallmatrix}}_{\begin{smallmatrix} \text{centre }  \\ \Pic^{e-2}(C) \times C^{(2)}  \end{smallmatrix}} & & \cP^2_{e} & \cdots & \cP^{m-1}_{e}\ar@{-->}[rr]^{\begin{smallmatrix} \text{type} \\(m-1,e+g-2(m+1)) \end{smallmatrix}}_{\begin{smallmatrix} \text{centre }   \\ \Pic^{e-m}(C) \times C^{(m)}  \end{smallmatrix}} & & \cP^m_{e} \ar[d]^{\pi} \\ \Pic^e(C) & & &&&& & & \cN. }\]

We can easily calculate the motivic Poincar\'{e} polynomial of $\cP^i_{e}$.

\begin{lemma}
	\label{lemma:MotivePairK0}
	We have
	\[ \chi(\cP^i_{e}) = \sum_{j=0}^i \chi(\Pic^{e-j}(C) \times C^{(j)}) \frac{\LL^{e+g - 2j -1} - \LL^j}{\LL - 1}. \]
\end{lemma}
\begin{proof}
	This follows directly from Theorem \ref{thm pairs VGIT} (see the above diagram).
\end{proof}

\begin{cor}
	\label{cor:MotivePairSmalli}
	When $3i\leq e+g-1$, we have an isomorphism in $\CHM^{\eff}(\k, \QQ)$:
	\[ \fh(\cP^i_{e}) \simeq \fh(\Jac(C))\otimes \left(\bigoplus_{j=0}^i \bigoplus_{k=j}^{e+g-2j-2}\fh( C^{(j)})(k) \right).\]
\end{cor}
\begin{proof}
The change in integral Chow motives of smooth projective varieties under a standard flip is described by a recent result of Jiang \cite{Jiang19}.	Combined with Theorem \ref{thm pairs VGIT} (note that the canonical class is always increasing along the flips, under the numerical hypothesis), we get the formula in the statement, even in $\CHM(\k, \ZZ)$.

We give here a quick proof without using the somewhat difficult result of Jiang. Note that the motives in the formula in Lemma \ref{lemma:MotivePairK0} are all Kimura finite dimensional and all coefficients are positive, when $3i\leq e+g-1$. Hence we can apply Corollary \ref{cor lifting effective identities} to deduce that 
	\[\fh(\cP^i_{e}) \simeq \bigoplus_{j=0}^i \fh\left(\Pic^{e-j}(C) \times C^{(j)}\right)\otimes \left(\bigoplus_{k=j}^{e+g-2j-2}\QQ(k)\right).\]
In order to get the desired form, it suffices to identify $\Pic^{e-j}(C)$ with $\Jac(C)$.
\end{proof}

Observe that the formula in Lemma \ref{lemma:MotivePairK0} contains terms with negative coefficients, when $3i>e+g-1$. 
The main goal of the next two sections is to deal with this case. In fact, we will give two formulas for the rational Chow motive of $\cP_e^i$. Just as in the second proof of Corollary \ref{cor:MotivePairSmalli}, our strategy is to work in the Grothendieck ring of Chow motives and establish an expression of the motivic Poincar\'{e} polynomial of these pairs moduli spaces in terms of Kimura finite dimensional motives with \textit{positive} coefficients in order to apply Proposition \ref{prop:UpgradeK0toCHM} or Corollary \ref{cor lifting effective identities}.

\subsection{A formula in terms of the Jacobian motive of the curve}
Choosing a zero-cycle $z$  of degree 1 on $C$ (its existence is guaranteed by the hypothesis that $\Pic^1(C)\neq \emptyset$), we have \[\fh(C)\simeq \QQ\oplus \fh^1(C)\oplus \QQ(1),\] in the category $\CHM^{\eff}(\k, \QQ)$,
where $\fh^1(C):=(C, \Delta_C-z\times C-C\times z)$, which can be appropriately called the \textit{Jacobian motive} of $C$. In some sense, $\fh^1(C)$ is the most fundamental  indecomposable building block for objects in the tensor subcategory of $\CHM(\k, \QQ)$ generated by $\fh(C)$. The purpose of this section is to give a formula of the rational Chow motive of the pair moduli space $\cP_e^i$ in terms of $\fh^1(C)$.

The following facts will be the main ingredient:
\begin{lemma}[K\"unnemann \cite{Kunnemann}]
	\label{lemma:Kunnemann}
Let $g$ be the genus of $C$ and $b$ a positive integer. 
 \begin{enumerate}[label=\emph{(\roman*)}, leftmargin=0.7cm]
 	\item $\Sym^b(\fh^1(C))=0$ if $b>2g$.
 	\item For any $g \leq b\leq 2g$, we have $\Sym^b(\fh^1(C))\simeq \Sym^{2g-b}\fh^1(C)\otimes \QQ(b-g)$ in $\CHM^{\eff}(\k, \QQ)$. Hence 
 	$[\Sym^b\fh^1(C)]=[\Sym^{2g-b}\fh^1(C)]\cdot\LL^{b-g}$ in $K_0(\CHM^{\eff}(\k, \QQ))$.
 \end{enumerate}
\end{lemma}

\begin{prop}
	\label{prop:MotivePairK0bis}
	The following identity holds in $\widehat{K_0}(\CHM(\k,\QQ))$.
	\[\chi(\cP_e^i)=\chi(\Jac(C))\sum_{b=0}^i[\Sym^b(\fh^1(C))]\cdot \frac{(\LL^b-\LL^{e+g-1-2i})(1-\LL^{i-b+1})(1-\LL^{i-b+2})}{(1-\LL)^2(1-\LL^2)}\]
\end{prop}
\begin{proof}
	Note that in $\CHM(\k, \QQ)$,
	 \[\fh(C^{(j)})\simeq \Sym^j(\fh(C))\simeq\Sym^j(\QQ\oplus \fh^1(C)\oplus \QQ(1))\simeq \bigoplus_{a+b+c=j}\Sym^b\fh^1(C)\otimes \QQ(c).\]
	Taking their classes in $K_0(\CHM(\k, \QQ))$, we obtain that  \[\chi(C^{(j)})=\sum_{a+b+c=j} [\Sym^b\fh^1(C)]\cdot \LL^c.\]
After plugging this into the formula in Lemma \ref{lemma:MotivePairK0} and exchanging the two summations, we obtain that
\[
\chi(\cP_e^i)= \chi(\Jac(C))\sum_{b=0}^i[\Sym^b(\fh^1(C))] \cdot\sum_{j=b}^{i}\frac{(1-\LL^{j-b+1})(\LL^{j}-\LL^{e+g-2j-1})}{(1-\LL)^{2}}.
\]
We then compute
\begin{align*}
(1-\LL^{2}) \sum_{j=b}^{i}(1-\LL^{j-b+1})(\LL^{j}-\LL^{e+g-2j-1})= & \LL^{b}+\LL^{b+1}-\LL^{i+1}-\LL^{i+2}-\LL^{b+1}+\LL^{2i-b+3}\\ 
\vspace{-1cm}
&-\LL^{e+g-2i-1} +\LL^{e+g-2b+2}+\LL^{e+g-b-i}\\ 
& +\LL^{e+g-b-i+1}-\LL^{e+g-2b+2}-\LL^{e+g-2}\\
= &(\LL^{b}-\LL^{e+g-1-2i})(1-\LL^{i-b+1})(1-\LL^{i-b+2})  
\end{align*}  
which proves the claimed formula.
\end{proof}

\begin{defn}
	\label{def:Q}
	Fix $g\geq 2$.
Let $i, e, b$ be positive integers satisfying $b\leq i<e/2$. Define
\begin{equation}
	Q_{i, e, b}(T):= \frac{(T^b-T^{e+g-1-2i})(1-T^{i-b+1})(1-T^{i-b+2})}{(1-T)^2(1-T^2)} \in \ZZ[T],
\end{equation}
which is a polynomial with integral coefficients. The coefficients are positive if $e+g-1-2i> b$ and are negative  if $e+g-1-2i< b$.
\end{defn}

The following elementary result is crucial for our purpose:
\begin{lemma}
	\label{lemma:PositiveCoeff}
Assume that $e+g-1-2i<b\leq i< \lfloor \frac{e}{2}\rfloor \leq 2g-3$, which implies that $b\geq g+1$. The polynomial 
\[R_{i,e,b}(T):=T^{b-g}Q_{i,e,b}(T)+Q_{i,e,2g-b}(T)\]
has positive coefficients.
\end{lemma}
\begin{proof}
	Define a partial order $\geq $ on $\ZZ[T]$ by  claiming that $P(T)\geq 0$ if and only if all coefficients of $P(T)$ are positive. Then
	\begin{align*}
		&T^{b-g}Q_{i,e,b}(T)+Q_{i,e,2g-b}(T)\\
		&=\frac{(T^{2g-b}-T^{e+g-1-2i})(1-T^{i-2g+b+1})(1-T^{i-2g+b+2})-T^{b-g}(T^{e+g-1-2i}-T^b)(1-T^{i-b+1})(1-T^{i-b+2})}{(1-T)^2(1-T^2)} \\
		&\geq \frac{(T^{2g-b}-T^{e+g-1-2i})(1-T^{i-2g+b+1})(1-T^{i-b+2})-T^{b-g}(T^{e+g-1-2i}-T^b)(1-T^{i-b+1})(1-T^{i-b+2})}{(1-T)^2(1-T^2)} \\
		&= \frac{(1-T^{i-b+2})}{(1-T)^2(1-T^2)}\left[(T^{2g-b}-T^{e+g-1-2i})(1-T^{i-2g+b+1})-T^{b-g}(T^{e+g-1-2i}-T^b)(1-T^{i-b+1})\right] \\
		&\geq \frac{(1-T^{i-b+2})}{(1-T)^2(1-T^2)}\left[(T^{2g-b}-T^{g})(1-T^{i-2g+b+1})-T^{b-g}(T^{g}-T^b)(1-T^{i-b+1})\right] \\
		&= \frac{(1-T^{i-b+2})T^{2g-b}}{(1-T)^2(1-T^2)}\left[(1-T^{b-g})(1-T^{i-2g+b+1})-(T^{2b-2g}-T^{3b-3g})(1-T^{i-b+1})\right] \\
		&= \frac{(1-T^{i-b+2})T^{2g-b}(1-T^{b-g})(1-T^{2b-2g})}{(1-T)^2(1-T^2)} \\
		&\geq 0,
	\end{align*}
	where the first inequality uses $b>g$ and the second inequality uses $e-2i\geq  1$. 
\end{proof}

In the sequel, for a polynomial $P(T)=\sum_{j=0}^n m_jT^j$ with positive integral coefficients, we denote by $P(\QQ(1))$ the effective Tate motive $\bigoplus_{j=0}^n \QQ(j)^{\oplus m_j}$.
\begin{cor}\label{cor Chow motives pairs}
	Assume that $i<\lfloor \frac{e}{2}\rfloor\leq 2g-3$,
	then the rational Chow motive of $\cP_e^i$ is as follows:
	\begin{enumerate}[label=\emph{(\roman*)}, leftmargin=0.7cm]
		\item If $3i\leq e+g-1$, then 
		\begin{equation*}
			\fh(\cP_e^i)\simeq \fh(\Jac(C))\otimes \left(\bigoplus_{b=0}^i \Sym^b(\fh^1(C))\otimes Q_{i,e,b}(\QQ(1))\right).
		\end{equation*}
		\item If $3i> e+g-1$, then $0<2g-i\leq g-e+2i<g<g+e-2i\leq i$, and 
		\begin{equation*}
						\fh(\cP_e^i)\simeq \fh(\Jac(C))\otimes \left(\bigoplus_{\substack{b<2g-i \text{ or}\\ |b-g|<e-2i}} \Sym^b(\fh^1(C))\otimes Q_{i,e,b}(\QQ(1))\oplus \bigoplus_{b=2g-i}^{g-e+2i} \Sym^b(\fh^1(C))\otimes R_{i,e,2g-b}(\QQ(1))\right).
		\end{equation*}
	\end{enumerate}
		Here $Q_{i,e,b}$ is the polynomial in Definition \ref{def:Q} and $R_{i,e,2g-b}$ is the polynomial defined in Lemma~\ref{lemma:PositiveCoeff}.
\end{cor}
\begin{proof}
	By Proposition \ref{prop:UpgradeK0toCHM}, it suffices to check that both sides agree in $K_0(\CHM(\k, \QQ))$. For case (i), it follows directly from Proposition \ref{prop:MotivePairK0bis} and the definition of $Q_{i,e,b}$ in Definition~\ref{def:Q}. For case (ii), the non-positive part in the expression of $\chi(\cP_e^i)$ in Proposition \ref{prop:MotivePairK0bis} is
	\[\chi(\Jac(C))\sum_{b=g+e-2i}^i[\Sym^b(\fh^1(C))]\cdot Q_{i,e,b}(\LL).\]
	It can be absorbed by the following effective part of $\chi(\cP_e^i)$ with $b$ running in the symmetric range with respect to $g$:
		\[\chi(\Jac(C))\sum_{b=2g-i}^{g-e+2i}[\Sym^b(\fh^1(C))]\cdot Q_{i,e,b}(\LL).\]
		Indeed, for any $b\in [e+g-2i, i]$,  by Lemma \ref{lemma:Kunnemann},
		\begin{align*}
		&[\Sym^b(\fh^1(C))]\cdot Q_{i,e,b}(\LL)+[\Sym^{2g-b}(\fh^1(C))]\cdot Q_{i,e,2g-b}(\LL)\\
		&=[\Sym^{2g-b}(\fh^1(C))]\LL^{b-g}\cdot Q_{i,e,b}(\LL)+[\Sym^{2g-b}(\fh^1(C))]\cdot Q_{i,e,2g-b}(\LL)\\
		&=[\Sym^{2g-b}(\fh^1(C))]\cdot R_{i,e,b}(\LL),
		\end{align*}
		which is a positive integral linear combination of effective motives by Lemma \ref{lemma:PositiveCoeff}.
\end{proof}

\subsection{A formula in terms of symmetric powers of the curve and its Jacobian}
In this section, we compute the rational Chow motive of $\cP_e^i$ again, but in terms of symmetric powers $C^{(k)}$ and $\Jac(C)$. The following basic fact will be our key tool:

\begin{lemma}[{\cite[Proposition 1.6(i)]{GL}}]
	\label{lemma:ReduceSymmetricPowers}
	For any $g\leq j\leq 2g-2$, we have the following identity in $K_0(\CHM(\k, \QQ))$:
	\[\chi(C^{(j)})=\chi(C^{(2g-2-j)})\cdot \LL^{j+1-g}+\chi(\Jac(C))\chi(\PP^{j-g}).\]
\end{lemma}

\begin{rmk}
	 Corollary \ref{cor lifting effective identities} allows one to  lift the equality in Lemma \ref{lemma:ReduceSymmetricPowers}  to an isomorphism in the category of  rational Chow motives. Note that the result for integral Chow groups is recent obtained by Jiang \cite[Section 5.1]{Jiang19}.
\end{rmk}

In the sequel, we impose the numerical constraint that 
\[	2i< e\leq 4g-5,\]
which will be sufficient for all the pair moduli spaces appearing in Theorem \ref{main thm H}. In view of Corollary \ref{cor:MotivePairSmalli}, we assume furthermore that 
\[3i>e+g-1,\]
which implies that $i\geq g$ and $e\geq 2g+1$. To summarise,  from now on, we work under the following hypothesis 
\begin{equation}
\label{eq:NumercialConstraints}
3i\geq e+g \quad \text{  and  }\quad 2g\leq 2i<e\leq 4g-5.
\end{equation}

\begin{prop}
	\label{prop:MotivePairK0ter}
	Under the above numerical assumptions, 
	the following identity holds in $\widehat{K_0}(\CHM^{\eff}(\k,\QQ))$ :
	\begin{align*}
	\chi(\cP_e^i)= & \:\chi(\Jac(C))\chi(C^{(g-1)})\frac{\LL^{g-1}-\LL^{e-g+1}}{1-\LL}+\chi(\Jac(C))\sum_{k=0}^{2g-3-i}\chi(C^{(k)})\frac{\LL^k-\LL^{e+g-2k-1}}{1-\LL}\\
	&+\chi(\Jac(C))\sum_{k=2g-2-i}^{g-2}\chi(C^{(k)})(\LL^{3g-3-2k}+\LL^{k})\frac{1-\LL^{e-2g+2}}{1-\LL}\\
	&+\chi(\Jac(C)^2)\frac{\LL^g(1-\LL^{e-2i-1})(1-\LL^{i+1-g})(1-\LL^{i+2-g})}{(1-\LL)^2(1-\LL^2)}.
	\end{align*}
\end{prop}
\begin{proof}
Note that under these assumptions, we have $g \leq i \leq 2g -2$. Therefore, we apply Lemma \ref{lemma:ReduceSymmetricPowers} to the symmetric powers $C^{(j)}$ in the formula in Lemma~\ref{lemma:MotivePairK0} with $g \leq j \leq i$. It is then a straightforward computation to obtain the following formula
	\begin{align*}
	\chi(\cP_e^i)=& \:\chi(\Jac(C))\sum_{k=0}^{g-1}\chi(C^{(k)})\frac{\LL^k-\LL^{e+g-2k-1}}{1-\LL}\\
	&+\chi(\Jac(C))\sum_{k=2g-2-i}^{g-2}\chi(C^{(k)})\frac{\LL^{3g-3-2k}-\LL^{e-2g+2+k}}{1-\LL}\\
	&+\chi(\Jac(C)^2)\sum_{k=0}^{i-g}\frac{(1-\LL^{k+1})(\LL^{k+g}-\LL^{e-g-1-2k})}{(1-\LL)^2}\\
	=& \:\chi(\Jac(C))\chi(C^{(g-1)})\frac{\LL^{g-1}-\LL^{e-g+1}}{1-\LL}+\chi(\Jac(C))\sum_{k=0}^{2g-3-i}\chi(C^{(k)})\frac{\LL^k-\LL^{e+g-2k-1}}{1-\LL}\\
	&+\chi(\Jac(C))\sum_{k=2g-2-i}^{g-2}\chi(C^{(k)})\left(\frac{\LL^{3g-3-2k}-\LL^{e-2g+2+k}}{1-\LL}+\frac{\LL^k-\LL^{e+g-2k-1}}{1-\LL}\right)\\
	&+\chi(\Jac(C)^2)\sum_{k=0}^{i-g}\frac{(1-\LL^{k+1})(\LL^{k+g}-\LL^{e-g-1-2k})}{(1-\LL)^2}.\\
	\end{align*}
	Note that 
	\[\frac{\LL^{3g-3-2k}-\LL^{e-2g+2+k}}{1-\LL}+\frac{\LL^k-\LL^{e+g-2k-1}}{1-\LL}=(\LL^{3g-3-2k}+\LL^{k})\frac{1-\LL^{e-2g+2}}{1-\LL}.\]
	It remains to check that 
	\[\sum_{k=0}^{i-g}\frac{(1-\LL^{k+1})(\LL^{k+g}-\LL^{e-g-1-2k})}{(1-\LL)^2}=\frac{\LL^g(1-\LL^{e-2i-1})(1-\LL^{i+1-g})(1-\LL^{i+2-g})}{(1-\LL)^2(1-\LL^2)},\]
	which follows as in the computation at the end of Proposition \ref{prop:MotivePairK0bis}.
\end{proof}

We are now able to prove Theorem \ref{main thm P}.

\begin{proof}[Proof of Theorem \ref{main thm P}]
The formula in the first case is just a reformulation of Corollary \ref{cor:MotivePairSmalli}. In the second case, we apply Proposition \ref{prop:UpgradeK0toCHM} to Proposition \ref{prop:MotivePairK0ter}. 
\end{proof}

\section{Motives of moduli spaces of rank 3 Higgs bundles}
\label{sec:MotiveRk3Higgs}

\subsection{Higgs moduli spaces}

We recall that a \emph{Higgs bundle} on $C$ is a pair $(E,\Phi)$ consisting of a vector bundle $E$ and a \emph{Higgs field} $\Phi\colon E\to E\otimes \omega_C$ which is a $\mathcal{O}_C$-linear homomorphism. There is a notion of semistability for Higgs bundles which involves verifying an inequality of slopes for Higgs subbundles (i.e. vector subbundles which are invariant under the Higgs field). This enables the construction of the moduli space $\cM=\cM(n,d)$ of semistable rank $n$ degree $d$ Higgs bundles on $C$ as a quasi-projective variety via Geometric Invariant Theory. In the case where $n$ and $d$ are coprime, semistability and stability coincide and $\cM$ is smooth; moreover $\cM$ is a non-compact hyper-K\"{a}hler manifold over $\k = \CC$ \cite{Hitchin}.

\subsection{The scaling action}

There is a $\GG_m$-action on $\cM$ given by scaling the Higgs field which was used by Hitchin \cite{Hitchin} and Simpson \cite{Simpson} to study the geometry of $\cM$. They established that this $\GG_m$-action is semi-projective in the sense that
\begin{itemize}
\item the fixed locus is proper and 
\item the limit of $t \in \GG_m$ acting on any point in $\cM$ exists as $t$ tends to  zero;
\end{itemize}
for details on the proof, see \cite[Section 9]{HT_MirrorSym}. The flow under this action induces an associated Bia{\l}ynicki-Birula decomposition \cite{BB} and, when $\cM$ is smooth, this is a deformation retract to the fixed locus. Consequently, the Voevodsky motive of $\cM$ is pure \cite[Corollary 6.9]{HPL_Higgs} and we can consider its associated Chow motive (for details, see \cite[Theorem 2.4 and $\S$6.1.1]{FHPL-rank2}). 

For a point $(E,\Phi) \in \cM$ fixed by the $\GG_m$-action, either the Higgs field $\Phi$ is zero (in which case the underlying vector bundle is a semistable vector bundle) or the Higgs field is non-zero and so there is $\GG_m \subset \Aut(E)$ inducing a weight decomposition $E = \oplus_i E_i$ and with respect to this decomposition $\Phi$ is given by non-trivial homomorphisms $E_i \ra E_{i+1} \otimes \omega_C$. Consequently, if $\Phi \neq 0$, we obtain a chain of vector bundle homomorphisms
\[ E_{i_0} \ra E_{i_0+1} \otimes \omega_C \ra E_{i_0+2} \otimes \omega_C^{\otimes 2} \ra \cdots \ra E_{i_0 + m} \otimes \omega_C^{\otimes m}. \]
The fixed components with $\Phi \neq 0$ are then indexed by the discrete invariants of this chain (this is equivalent to fixing the ranks and degrees of the $E_i$). 

Once one fixes discrete invariants for chains (i.e.~a tuple $\underline{n}$ and $\underline{d}$ corresponding to the ranks and degrees of the vector bundles in the chain), one can construct projective moduli spaces of chains which are semistable with respect to a stability parameter $\alpha$ (a tuple of real numbers indexed by the vector bundles in the chain) via Geometric Invariant Theory \cite{schmitt_moduli}. The deformation theory and wall-crossing for chains were studied in \cite{ACGPS}, as well as the relationship between stability for chains and Higgs bundles. In particular, the connected components of the fixed point set of the $\GG_m$-action on $\cM$ are moduli spaces of $\alpha_H$-semistable chains for different discrete invariants, where $\alpha_H$ is a Higgs stability parameter satisfying $\alpha_{H,i} - \alpha_{H,i+1} = 2g-2$ for all $i$ (see \cite{Simpson} and \cite[Corollary 2.6]{HPL_Higgs}). Provided $n$ and $d$ are coprime, the Higgs stability parameter is generic for the discrete invariants for chains appearing as fixed loci components (i.e semistability and stability coincide and these chain moduli spaces are smooth projective varieties). Hence, the fixed locus consists of the moduli space $\cN$ of semistable vector bundles and moduli spaces of $\alpha_H$-semistable chains with various discrete invariants.

For small values of $n$ and for values of $d$ coprime to $n$, this $\GG_m$-action has been used to calculate the Poincar\'{e} polynomial of $\cM$ in rank $n=2$ by Hitchin \cite{Hitchin} and in rank $n=3$ by Gothen \cite{Gothen}; in these low rank computations, the fixed loci are related to symmetric powers of $C$, the Jacobian of $C$ and moduli spaces of pairs (consisting of a vector bundle and a section) studied by Bradlow \cite{Bradlow} and Thaddeus \cite{Thaddeus_pairs}, which depend on a stability parameter.

This scaling action is also used in the study the class of $\cM$ in the Grothendieck ring of varieties in \cite{GPHS} and the Voevodsky motive of $\cM$ in \cite{HPL_Higgs}, where wall-crossing for chains plays an important role. In rank $n = 2$ and odd degree, we obtained a formula for the integral motive of $\cM$ in \cite[Theorem 1.4]{FHPL-rank2} in terms of $\cN$.

\subsection{The fixed loci and their motives in rank 3}
In order to  compute the motive of the Higgs bundle moduli space in rank $3$ and coprime degree $d$, we will use the motivic Bia{\l}ynicki-Birula decomposition associated to the $\GG_m$-action on $\cM$ and Gothen's description of the fixed components \cite{Gothen}. Since the motive of $\cM$ is pure, we can view the motivic Bia{\l}ynicki-Birula decomposition as an isomorphism in $\CHM(\k,\QQ)$. For this calculation, we need to describe the motive of each component in the fixed locus. By specifying the ranks of the vector bundles in the chain, we split the possible fixed components into the following types:\newline

\noindent \textbf{Type (1,1,1):} $(E,\Phi)$ with $E = L_1 \oplus L_2 \oplus L_3$ a sum of three line bundles and $\Phi_i : L_i \ra L_{i+1} \otimes \omega_C$,\newline

\noindent \textbf{Type (1,2):} $(E,\Phi)$ with $E = L \oplus F$ for a line bundle $L$ with $\Phi : L \ra F \otimes \omega_C$,\newline

\noindent \textbf{Type (2,1):} $(E,\Phi)$ with $E = F \oplus L$ for a line bundle $L$ with $\Phi : F \ra L \otimes \omega_C$,\newline

\noindent \textbf{Type (3):} $(E,\Phi)$ with $\Phi = 0$ and $E$ a rank 3 semistable vector bundle.\newline

In the last case, Type (3), there is only one fixed component: the moduli space $\cN = \cN_C(3,d)$ of semistable vector bundles of rank $n=3$ and degree $d$ on $C$. For all the other types, there are different components indexed by the possible degrees of the vector bundles in the chain. We modify Gothen's description of the fixed components of each type for Higgs bundles with fixed determinant to arbitrary determinant in the following sections.

The Chow motive of the Higgs moduli space will be, via the motivic Bia{\l}ynicki-Birula decomposition, a sum of motives corresponding to each type:
\begin{equation}\label{motivic BB}
\fh(\cM_C(3,d)) = \fh(\cN_C(3,d)) \oplus M_{(1,1,1)} \oplus M_{(2,1)} \oplus M_{(1,2)}
\end{equation}
where the Chow motives $M_{(1,1,1)}$,  $M_{(2,1)}$ and $M_{(1,2)}$ will be each computed in turn below. Note that a priori, these motives depend on $d$. Each of these motives is a direct sum of Tate twists of the motives of the fixed components of that type. The Tate twist appearing with a fixed component $F$ is the codimension of the corresponding Bia{\l}ynicki-Birula stratum $F^+$ (which retracts onto $F$ via the downwards flow); since the downward flow is Lagrangian, 
\[ \codim(F^+)= \frac{1}{2} \dim(\cM(3,d)) - \dim(F) = 9(g-1) + 1 - \dim (F) .\]

\subsubsection{Fixed loci of Type (1,1,1)}

Suppose that $(E,\Phi)$ is a stable Higgs bundle fixed by the scaling action of Type (1,1,1); then $E = L_1 \oplus L_2 \oplus L_3$ and
\[ \Phi = \left( \begin{array}{ccc} 0 & 0 & 0 \\ \phi_1 & 0 & 0 \\ 0 & \phi_2 &  0 \end{array} \right) \]
where $\phi_i : L_i \ra L_{i+1} \otimes \omega_{C}$ are non-zero. Let $l_i := \deg (L_i)$; then $d = l_1 + l_2 + l_3$. The non-zero homomorphisms $\phi_i$ correspond to non-zero sections of the line bundle $M_i = L_i^{-1} \otimes L_{i+1} \otimes \omega_C$, which has degree $m_i = l_{i+1} - l_i + 2g-2 \geq 0$. Given $M_1$ and $M_2$ and one of the three line bundles (say $L_2$), we can determine the other two line bundles ($L_1=L_2 \otimes M_1^{-1} \otimes \omega_{C}$ and $L_2 = L_3 \otimes M_2^{-1} \otimes \omega_C$). Hence, the fixed points $(L_1 \oplus L_2 \oplus L_3,\Phi)$ with degrees $(l_1,l_2,l_3)$ are parametrised by
\[ \Pic^{l_1}(C) \times C^{(l_2 - l_1 + 2g-2)} \times C^{(d-2l_2 - l_1 + 2g-2)}.\]
The possible ranges for $(l_1,l_2)$ are determined by $\phi_i$ being non-zero and stability of $(E,\Phi)$: the $\Phi$-invariant subbundles of $E$ are $L_2 \oplus L_3$ and $L_3$ and stability of $(E,\Phi)$ implies
\[  \frac{l_2 + l_3}{2} < \frac{d}{3} \quad \text{and} \quad l_3 < \frac{d}{3}. \]
If we equivalently phrase these inequalities in terms of $m_i$, we get
\[ 2m_1 + m_2 < 6g-6 \quad \text{and} \quad m_1 + 2m_2 < 6g - 6 \]
and we recover the $l_i$ by noting that $3 l_2 = d +m_1 - m_2$. Hence, we need $d \equiv m_2 - m_1 $ modulo $3$; we note that this last constraint was accidentally omitted by Gothen and consequently resulted in an error in the formula for the Poincar\'{e} polynomial stated in \cite{Gothen}; this was later corrected in \cite[$\S$10]{HT_MirrorSym} and in the formula for the virtual motivic class in \cite[Appendix]{GPHS}.

\begin{cor} 
	\label{cor:FixLoci(1,1,1)}
The fixed locus components of Type (1,1,1) are given by 
\[ F_{(1,1,1)}=\bigsqcup_{\begin{smallmatrix} (m_1,m_2) \in \NN^2 \\ \max(2m_1+m_2,m_1 + 2m_2) < 6g-6 \\ d \equiv m_2 - m_1 \mod 3 \end{smallmatrix} } \Pic^{(d+ m_1 -m_2)/3}(C) \times C^{(m_1)}\times C^{(m_2)}. \]
Thus the contribution $M_{(1,1,1)}$ in \eqref{motivic BB} to the Chow motive of the Higgs bundle moduli space is
\[ M_{(1,1,1)} = \bigoplus_{\begin{smallmatrix} (m_1,m_2) \in \NN^2 \\ \max(2m_1+m_2,m_1 + 2m_2) < 6g-6 \\ d \equiv m_2 - m_1 \mod 3 \end{smallmatrix} } \fh(\Pic^{(d+ m_1 -m_2)/3}(C))  \otimes \fh(C^{(m_1)})\otimes \fh(C^{(m_2)})(8g-8 - m_1 -m_2). \]  
 \end{cor}

 \begin{rmk} 
A priori $F_{(1,1,1)}$ depends on $d$ and so we should really write $F(d)_{(1,1,1)}$. However, assuming that $C$ has a degree $1$ line bundle, the map $(m_1,m_2) \mapsto (m_2,m_1)$ determines an isomorphism
\[ F(d)_{(1,1,1)} \cong F(-d)_{(1,1,1)}.\]
\end{rmk}

\subsubsection{Fixed loci of Type (1,2)}\label{sec Type 12} Suppose that $(E,\Phi)$ is a stable Higgs bundle fixed by the scaling action of Type (1,2); then $E = L \oplus F$ and $\Phi$ is determined by a non-zero homomorphism $L \ra F \otimes \omega_C$. Equivalently, we can think of this non-zero homomorphism as a non-zero section $\phi$ of the vector bundle $V := L^{-1} \otimes F \otimes \omega_C$. Gothen bounded the possible values of $l:= \deg(L)$ as follows
\begin{equation}\label{ineq l} 
\frac{d}{3} < l < \frac{d}{3} + g-1
\end{equation}
using stability of $(E,\Phi)$ and the fact that the homomorphism $L \ra F \otimes \omega_C$ is non-zero. Hence $(E,\Phi)$ determines a line bundle $L$ of degree $l$ and a pair $(V,\phi)$ consisting of a rank $2$ vector bundle of degree $e:=d - 3l +4g-4$ and a non-zero section. Gothen \cite{Gothen} showed that Higgs stability of $(E,\Phi)$ corresponds to pair stability of $(V,\phi)$ for a particular value of the stability parameter $\sigma$. 

\begin{prop}[{\cite[Proposition 2.5]{Gothen}}]
The component of the $\GG_m$-fixed locus consisting of Higgs bundles $(E = L \oplus F, \Phi: L \ra F \otimes \omega_C)$ of Type (1,2) with $\deg(L) = l$ satisfying \eqref{ineq l} is isomorphic to the product $\Pic^l(C) \times \cP^{\sigma_{d,l}-ss}(2,e_{d,l})$ where $e_{d,l}:= d-3l +4g-4$ and $\sigma_{d,l} = \frac{l}{2} - \frac{d}{6}$.
\end{prop}

Consequently, we obtain the following description of the fixed locus components of Type (1,2) and the motivic contribution $M_{(1,2)}$ in \eqref{motivic BB}.

\begin{cor} 
	\label{cor:FixLoci(1,2)}
The fixed locus components in $\cM^{\GG_m}$ of Type (1,2) are given by
\[ 
F_{(1,2)}=\bigsqcup_{\frac{d}{3} < l < \frac{d}{3} + g-1} \Pic^l(C)\times \cP^{\sigma_{d,l}-ss}(2,e_{d,l}) \]
 where $e_{d,l}:= d-3l +4g-4$ and $\sigma_{d,l} = \frac{l}{2} - \frac{d}{6}$. Hence, the contribution $M_{(1,2)}$ in \eqref{motivic BB} to the Chow motive of the Higgs bundle moduli space is 
 \[ M_{(1,2)} = \bigoplus_{\frac{d}{3} < l < \frac{d}{3} + g-1} \fh(\Pic^l(C))\otimes \fh(\cP^{\sigma_{d,l}-ss}(2,e_{d,l}))(2g - 2 -  d+ 3l). \]
 \end{cor}

If we make a change of variables $k = l - \lfloor \frac{d}{3} \rfloor -1$, then $3l - d = 3k + 3 -x$ where $x \in \{1,2\}$ satisfies $x \equiv d$ mod $3$. Hence, we have
 \begin{equation}\label{eq motive (1,2)}
  M_{(1,2)} = \bigoplus_{k=0}^{g-2} \fh(\Pic^{k+1 + \frac{d-x}{3}}(C))\otimes \fh(\cP^{(\frac{k+1}{2} - \frac{x}{6})-ss}(2,4g -3k -7+x))(2g +3k +1 -x).
 \end{equation}
For $0 \leq k \leq g-2$, pairs of degree $e(k,x):=4g -3k -7+x$ have different moduli spaces $\cP^j_{e(k,x)}$ corresponding to the different chambers $C_j$ in which the stability parameter lies (see Theorem \ref{thm pairs VGIT}). The stability parameter $\frac{k+1}{2} - \frac{x}{6}$ lies in the chamber $C_{i_k}$ where $i_{k} = 2g -2k - 5 +x$ and so 
 \[ \cP^{(\frac{k+1}{2} - \frac{x}{6})-ss}(2,4g -3k -7+x)) \cong \cP^{2g -2k - 5 +x}_{4g -3k -7+x} \]
 and the Chow motive of this pairs moduli space is calculated by Corollary \ref{cor Chow motives pairs} above.

\subsubsection{Fixed loci of Type (2,1)}\label{sec Type 21} The dual of a Higgs bundle $(E,\Phi)$ is given by $(E^\vee,\Phi^\vee \otimes \mathrm{Id}_{\omega_C})$ and preserves stability. Furthermore, the dual of a stable Higgs bundle of Type (2,1) is a stable Higgs bundle of Type (1,2). More precisely, if we write the fixed loci as depending on $d$ as $F(d)_{(1,2)}$ and $F(d)_{(2,1)}$, then dualising gives an isomorphism $F(d)_{(2,1)} \cong F(-d)_{(1,2)}$. Therefore, the fixed components of Type (2,1) with degree $d$ are indexed by 
\begin{equation}\label{ineq j} 
\frac{d}{3} + 1-g < j < \frac{d}{3}
\end{equation}

\begin{prop} [{\cite[Proposition 2.9]{Gothen}}]
The component of the $\GG_m$-fixed locus consisting Higgs bundles $(E =  F \oplus L, \Phi: F \ra L \otimes \omega_C)$ of Type (2,1) with $\deg(L) = j$ satisfying \eqref{ineq j} is isomorphic to the product $\Pic^j(C) \times \cP^{\sigma_{d,j}-ss}(2,f_{d,j})$ where $f_{d,j}:=3j-d +4g-4$ and $\sigma_{d,j} :=  \frac{d}{6} - \frac{j}{2}$.
\end{prop}

\begin{cor} 
	\label{cor:FixLoci(2,1)}
The fixed locus components of Type (2,1) are 
\[ 
F_{(2,1)}=\bigsqcup_{\frac{d}{3} + 1-g < j < \frac{d}{3} } \Pic^j(C)\times \cP^{\sigma_{d,j}-ss}(2,f_{d,j}) \]
 where $f_{d,j}:= 3j -d +4g-4$ and $\sigma_{d,j} :=  \frac{d}{6} - \frac{j}{2}$. Hence, the contribution $M_{(2,1)}$ in \eqref{motivic BB} to the Chow motive of the Higgs bundle moduli space is 
 \[ M_{(2,1)} = \bigoplus_{\frac{d}{3} + 1-g < j < \frac{d}{3} } \fh(\Pic^j(C))\otimes \fh(\cP^{\sigma_{d,j}-ss}(2,f_{d,j}))(2g - 2 - 3j + d). \]
 \end{cor}
 
If we make a change of variables $k =  \lfloor \frac{d}{3} \rfloor -j$, then $d - 3j = 3k + x$ where $x \in \{1,2\}$ satisfies $x \equiv d$ mod $3$. Then we have
\begin{equation}\label{eq motive (2,1)}
 M_{(2,1)} = \bigoplus_{k=0}^{g-2} \fh(\Pic^{\frac{d-x}{3} - k}(C))\otimes\fh(\cP^{(\frac{k}{2}+\frac{x}{6})-ss}(2,4g-3k -4-x))(2g +3k -2 +x).
\end{equation}
For $0 \leq k \leq g-2$ and degree $f(k,x):=4g-3k-4-x$, the pairs stability parameter $\frac{k}{2}+\frac{x}{6}$ lies in the chamber $C_{i(k)}$ where $i(k) = 2g -2k - 2 -x$ and so 
\[ \cP^{(\frac{k+1}{2} - \frac{x}{6})-ss}(2,4g -3k -7+x)) \cong \cP^{2g -2k - 2 -x}_{4g-4-3k-x}. \]
 

\subsection{The proof of the formula in rank 3}
In this section, by putting the results in the previous sections together, we compute the  Chow motive of $\cM_C(3, d)$, the moduli spaces of rank 3 stable Higgs bundles of degree $d$, which is coprime to 3.


\begin{proof}[Proof of Theorem \ref{main thm H}]
	Analogously to Remark \ref{rmk:IndependentOfDegree}, by taking dual Higgs bundles and tensoring with line bundles, we see that the isomorphism class of the moduli space $\cM(3, d)$ is independent of $d$ as long as $(3, d)=1$ and $\Pic^1(C)(\k)\neq \emptyset$. Therefore, without loss of generality, we can assume that $d=1$, and identify $\Pic^j(C)$ with $\Jac(C)$.
	
	We apply the motivic Bia{\l}ynicki-Birula decomposition in the form of  \cite[Theorem A.4 (i)]{HPL_Higgs}: the statement  follows from Equation \eqref{motivic BB} and the combination of Corollaries \ref{cor:FixLoci(1,1,1)}, \ref{cor:FixLoci(1,2)} and \ref{cor:FixLoci(2,1)} (see also Equations \eqref{eq motive (1,2)} and \eqref{eq motive (2,1)} and note that we now take $d = 1$ so $x = 1$). 
\end{proof}



\subsection{Motives of Higgs moduli spaces with fixed determinant}

The rational Chow motive of $\cM(n,d)$ for any rank $n$ and coprime degree $d$ lies in in the thick tensor subcategory generated by the motive of $C$ \cite{HPL_Higgs}. As observed in the rank 2 case in \cite[Proposition 6.3]{FHPL-rank2}, for a line bundle $\cL$ of odd degree $d$ on a general complex curve $C$, the rational Chow motive of $\cM_{\cL}(2, d)$ is \textit{not} in the thick tensor subcategory generated by the motive of $C$. We have a similar phenomenon in any rank $n$.

\begin{prop}
	\label{prop:MotiveHiggsFixDetNotGeneratedByC}
If $C$ is a general complex curve, then for a line bundle $\cL$ of degree $d$ coprime to $n$, the rational Chow motive of $\cM_{\cL}(n, d)$ is \textit{not} in the thick tensor subcategory generated by the motive of $C$.
\end{prop}
\begin{proof}
Consider the fixed loci of the scaling action on $\cM_{\cL}=\cM_{\cL}(n, d)$. Similarly to Gothen's description in rank 3, we have fixed components of Type $(1,\dots ,1)$ consisting of Higgs bundles $(E,\Phi)$ where $E = \oplus_{i=1}^n L_i$ is a sum of line bundles and the Higgs field is given by non-zero homomorphisms $\phi_i : L_i \ra L_{i+1} \otimes \omega_{C}$. Without fixing the determinant, this data is parametrised by a Picard group and $n-1$ copies of symmetric powers of $C$ (one symmetric power for each $\phi_i$). Fixing the determinant, results in fixed loci of Type $(1,\dots, 1)$ the form 
\[ \widetilde{C^{(m_1)}\times \cdots  \times  C^{(m_{n-1})} }, \]
which is the degree $n^{2g}$ \'etale cover of $C^{(m_1)}\times \cdots \times C^{(m_{n-1})}$ obtained as the pullback of the isogeny $\Jac(C)\xrightarrow{\cdot n} \Jac(C)$. For the fixed determinant case, these fixed loci are discussed in rank $n=2$ in \cite{Hitchin} and $n = 3$ in \cite{Gothen}. The range of possible values of $m_i$ is bounded by degree reasons and stability of the Higgs pair $(E,\Phi)$. In particular (taking all but one $m_i$ to be zero and $m_i = 1$), the motive of $\tilde{C}$ appears as a direct summand in the motive of $\cM_{\cL}$ by the motivic Bia{\l}ynicki-Birula decomposition. 

We claim that $\fh(\tilde{C})$ (hence  $\cM_{\cL}$ also) does not belong to the tensor subcategory generated by $\fh(C)$. Indeed, the same argument in  \cite[Proposition 6.3]{FHPL-rank2} shows that $\Jac(\tilde{C})$ is isogenous to the product $$\prod_{t\in H^1(C,  \ZZ/n\ZZ)} A_t,$$
	where $A_t$ is the complement abelian variety of $\Jac(C)$ in $\Jac(C_t)$ associated with the not necessarily connected \'etale  $n$-fold cover $C_t\to C$ determined by $t$. Therefore, it suffices to show that for some $t$, the Hodge structure $H^1(A_t, \QQ)$ is not in the tensor subcategory $\langle H^1(C, \QQ)\rangle^{\otimes}_{\QQ-\operatorname{HS}}$ of the category of rational Hodge structures generated by $H^1(C, \QQ)$. To this end, observe that for the $n$-fold covers of the form $A_t=(\coprod_{j=1}^{n-2} C)\coprod C'\to C$ with $C'\to C$ a connected \'etale double cover, $A_t$ is isogenous to $\Jac(C)^{n-2}\times\operatorname{Prym}(C'/C)$, where $\operatorname{Prym}(C'/C)$ is the Prym variety associated with the double cover $C'/C$. By an argument using the representation theory of speical Mumford--Tate groups in \cite[Proposition 6.3]{FHPL-rank2}, we know that $H^1(\operatorname{Prym}(C'/C),\QQ)\notin \langle H^1(C, \QQ)\rangle^{\otimes}_{\QQ-\operatorname{HS}}$ provided that $C$ is general. 
\end{proof}

\section{Examples and applications}\label{sec ex app}

In this section, we give formulas for the motives of $\cM_C(3,d)$ when $g=2,3$, as well as the resulting Hodge numbers (which were already implicitly known in rank 3 and arbitrary genus by \cite[Appendix]{GPHS}). The Hodge diamonds in this section were generated using Belmans' Hodge diamond cutter \cite{Belmans_HodgeDiamond}.

For ease of notation, for integers $m_1 < m_2 < \dots < m_r$, let 
\[ T_{m_1, m_2, \dots, m_r} =  \QQ(m_1)\oplus \QQ(m_2)\oplus \dots\oplus \QQ(m_r) \quad \text{and} \quad 
T_{[m_1,m_2]} =  \bigoplus_{i=m_1}^{m_2} \QQ(i).\]
We also let $J:=\Jac(C)$ denote the Jacobian of $C$.

\subsection{Genus 2 case}
For a curve $C$ of genus $g=2$, the rational Chow motive of $\cN_C(3,d)$ and $\cM_C(3, d)$, for $d$ coprime to 3, are given as follows.

\begin{align*}
	\fh(\cN_C(3,d))\simeq & \: \fh(J)\otimes \left(T_{0,8}\oplus \fh(C)\otimes T_{1,2,5,6}\oplus \fh(C^{(2)})\otimes T_{2,4}\oplus \fh(C^2)(3)\right);\\
	\\
	\fh(\cM_C(3,d))\simeq & \: \fh(J)\otimes \left(T_{0,8}\oplus \fh(C)\otimes T_{1,2,5,6,7}\oplus \fh(C^{(2)})\otimes T_{2,4,6}\oplus \fh(C^2)(3)\oplus \fh(C\times C^{(2)})(5)\right)\\
	&\oplus \fh(J^2)\otimes \left(T_{[3,6]}\oplus T_{[4,6]} \oplus \fh(C)(4)\right).
\end{align*}

The resulting Hodge numbers of the pure Hodge structure on the cohomology of $\cM(3,d)$ are given by:

\[
\left(\begin{array}{rrrrrrrrrrr}
1 & 2 & 1 & 0 & 0 & 0 & 0 & 0 & 0 & 0 & 0 \\
2 & 5 & 6 & 5 & 2 & 0 & 0 & 0 & 0 & 0 & 0 \\
1 & 6 & 16 & 22 & 16 & 6 & 1 & 0 & 0 & 0 & 0 \\
0 & 5 & 22 & 45 & 54 & 41 & 20 & 5 & 0 & 0 & 0 \\
0 & 2 & 16 & 54 & 104 & 126 & 96 & 44 & 12 & 2 & 0 \\
0 & 0 & 6 & 41 & 126 & 222 & 246 & 177 & 80 & 20 & 2 \\
0 & 0 & 1 & 20 & 96 & 246 & 390 & 390 & 239 & 82 & 12 \\
0 & 0 & 0 & 5 & 44 & 177 & 390 & 508 & 394 & 168 & 30 \\
0 & 0 & 0 & 0 & 12 & 80 & 239 & 394 & 369 & 184 & 38 \\
0 & 0 & 0 & 0 & 2 & 20 & 82 & 168 & 184 & 104 & 24 \\
0 & 0 & 0 & 0 & 0 & 2 & 12 & 30 & 38 & 24 & 6
\end{array}\right)
\]
\vspace{-0.2cm}
\subsection{Genus 3 case}
For a curve $C$ of genus $g=3$, the rational Chow motive of $\cN_C(3,d)$ and $\cM_C(3, d)$, for $d$ coprime to 3, are given as follows.
\vspace{-0.2cm}
\begin{align*}
\fh(\cN_C(3,d))\simeq \: \fh(J) \:\otimes & \: \Big(T_{0,16}\oplus \fh(C)\otimes T_{1,2,13,14}\oplus \fh(C^{(2)})\otimes T_{2,4, 10, 12}\oplus \fh(C^2)\otimes T_{3,11} \\
&\oplus \fh(C^{(3)})\otimes T_{3,6,7,10}\oplus \fh(C\times C^{(2)})\otimes T_{4,5,8,9}\\
&\oplus \fh(C^{(4)})\otimes T_{4,8} \oplus \fh(C\times C^{(3)})\otimes T_{5,7}\oplus \fh(C^{(2)}\times C^{(2)})(6)\Big);
\end{align*}
\vspace{-0.5cm}
\begin{align*}
\fh(\cM_C(3,d))\simeq & \: \fh(J)\otimes  \Big( T_{0,16}\oplus \fh(C)\otimes T_{1,2,13,14,15}\oplus \fh(C^{(2)})\otimes T_{2,4, 10, 12,14}\oplus \fh(C^2)\otimes T_{3,11} \\
& \quad \quad \quad \quad \oplus \fh(C^{(3)})\otimes T_{3,6,7,10}\oplus \fh(C\times C^{(2)})\otimes T_{4,5,8,9, 13}\\
& \quad \quad \quad \quad \oplus \fh(C^{(4)})\otimes T_{4,8,12} \oplus \fh(C\times C^{(3)})\otimes T_{5,7,12}\oplus \fh(C^{(2)}\times C^{(2)})(6)\\
& \quad \quad \quad \quad \oplus \fh(C^{(5)})(11)\oplus \fh(C^{(2)}\times C^{(3)})(11)\oplus \fh(C\times C^{(5)})(10)\\
& \quad \quad \quad \quad \oplus \fh(C^{(2)}\times C^{(4)})(10)  \oplus \fh(C^{(3)}\times C^{(4)})(9)\Big)\\
&\oplus \fh(J^2)\otimes  \Big(T_{[5,13]}\oplus T_{[6,13]}\oplus T_{[8,13]}\oplus T_{[9,13]}\oplus \fh(C)\otimes (T_{[6,11]}\oplus T_{[7,11]}\oplus T_{[9,11]})\\
& \quad \quad \quad \quad \quad \oplus \fh(C^{(2)})\otimes (T_{[8,9]}\oplus T_{[7,9]})\Big).
\end{align*}
The Hodge diamond of the pure Hodge structure of $\cM_C(3,d)$ for $g=3$ is already quite large and difficult to display properly. Since the motive $\fh(\cM_C(3,d))$ is a multiple of the motive $\fh(J)$ and the Hodge structure of $H^{*}(J)$ is well-known, we give the Hodge diamond of ``$\fh(\cM_C(3,d))/\fh(J)$'' (which coincides with the Hodge diamond of the moduli space of $\mathrm{PGL}_{3}$-Higgs bundles on a curve of genus $3$):
\vspace{-0.2cm}
\[\left(\begin{array}{rrrrrrrrrrrrrrrrr}
1 & 0 & 0 & 0 & 0 & 0 & 0 & 0 & 0 & 0 & 0 & 0 & 0 & 0 & 0 & 0 & 0 \\
0 & 1 & 3 & 0 & 0 & 0 & 0 & 0 & 0 & 0 & 0 & 0 & 0 & 0 & 0 & 0 & 0 \\
0 & 3 & 3 & 6 & 3 & 0 & 0 & 0 & 0 & 0 & 0 & 0 & 0 & 0 & 0 & 0 & 0 \\
0 & 0 & 6 & 13 & 12 & 12 & 1 & 0 & 0 & 0 & 0 & 0 & 0 & 0 & 0 & 0 & 0 \\
0 & 0 & 3 & 12 & 34 & 30 & 21 & 10 & 0 & 0 & 0 & 0 & 0 & 0 & 0 & 0 & 0 \\
0 & 0 & 0 & 12 & 30 & 63 & 78 & 45 & 21 & 3 & 0 & 0 & 0 & 0 & 0 & 0 & 0 \\
0 & 0 & 0 & 1 & 21 & 78 & 122 & 147 & 99 & 41 & 12 & 0 & 0 & 0 & 0 & 0 & 0 \\
0 & 0 & 0 & 0 & 10 & 45 & 147 & 242 & 261 & 195 & 80 & 21 & 3 & 0 & 0 & 0 & 0 \\
0 & 0 & 0 & 0 & 0 & 21 & 99 & 261 & 447 & 456 & 330 & 156 & 42 & 6 & 0 & 0 & 0 \\
0 & 0 & 0 & 0 & 0 & 3 & 41 & 195 & 456 & 731 & 777 & 537 & 251 & 72 & 12 & 1 & 0 \\
0 & 0 & 0 & 0 & 0 & 0 & 12 & 80 & 330 & 777 & 1151 & 1173 & 798 & 362 & 102 & 15 & 1 \\
0 & 0 & 0 & 0 & 0 & 0 & 0 & 21 & 156 & 537 & 1173 & 1659 & 1587 & 1020 & 417 & 102 & 12 \\
0 & 0 & 0 & 0 & 0 & 0 & 0 & 3 & 42 & 251 & 798 & 1587 & 2069 & 1776 & 990 & 324 & 45 \\
0 & 0 & 0 & 0 & 0 & 0 & 0 & 0 & 6 & 72 & 362 & 1020 & 1776 & 2003 & 1407 & 549 & 93 \\
0 & 0 & 0 & 0 & 0 & 0 & 0 & 0 & 0 & 12 & 102 & 417 & 990 & 1407 & 1167 & 537 & 102 \\
0 & 0 & 0 & 0 & 0 & 0 & 0 & 0 & 0 & 1 & 15 & 102 & 324 & 549 & 537 & 276 & 60 \\
0 & 0 & 0 & 0 & 0 & 0 & 0 & 0 & 0 & 0 & 1 & 12 & 45 & 93 & 102 & 60 & 15
\end{array}\right)
\]

\subsection{Chow groups}
\label{subsec:Chowgroups}
The formulas in Theorems \ref{thm:main thm N} and \ref{main thm H} directly imply formulas for the Chow groups of $\cN_C(3,d)$, $\cN_{C,\cL}(3,d)$ and $\cM_C(3,d)$ via the representability of Chow groups in the categories of Chow and Voevodsky motives; see Equation \eqref{eq Chow gps as Hom gps}. We make these formulas explicit in low (co)dimension. For simplicity, in order to exclude degenerate cases, we assume that $g\geq 2$.

The following results for $\cN_{C,\cL}(3,d)$ are similar to those in \cite[\S 4.3.1]{FHPL-rank2} in rank $2$.

\begin{cor}\
\begin{enumerate}[label=\emph{(\roman*)}, leftmargin=0.7cm]
\item $\CH^{1}(\cN_{C,\cL}(3,d))\simeq \ZZ$.
\item $\CH^{2}(\cN_{C,\cL}(3,d))_\QQ\simeq \QQ^{\oplus 2}\oplus \CH_{0}(C)_{\QQ}$.
\item $\CH^{3}(\cN_{C,\cL}(3,d))_\QQ\simeq \left\{ \begin{array}{ll} \QQ^{2}\oplus\CH_{0}(C)_{\QQ}\oplus \Pic(\Jac(C))_{\QQ} & \text{ if } g=2, \\ \QQ^{3}\oplus\CH_{0}(C)_{\QQ}\oplus \Pic(\Jac(C))_{\QQ} &   \text{ if }  g\geq 3. \end{array}\right.$
\item $\CH_{0}(\cN_{C,\cL}(3,d))_\QQ\simeq \QQ$.
\item $\CH_{1}(\cN_{C,\cL}(3,d))_\QQ\simeq \CH_{0}(C)_{\QQ}.$
\item $\CH_{2}(\cN_{C,\cL}(3,d))_\QQ\simeq \QQ \oplus \CH_{0}(C)_{\QQ}\oplus \CH_{0}(C^{(2)})_{\QQ}.$
\end{enumerate}
\end{cor}
\begin{proof}
  First of all, the variety $\cN_{C,\cL}(3,d)$ is a smooth projective Fano variety, hence its Picard group is torsion-free, which means that it suffices to prove (i) (and all the other statements) with rational coefficients.

By inspecting the indices in the formula of Theorem \ref{thm:main thm N}, we see that we can write
\begin{align*}
  \fh(\cN_{C,\cL}(3,d))= &\QQ(0)\oplus \fh(C)(1)\oplus (\fh(C)\oplus \fh(C^{(2)}))(2)\oplus (\fh(C^{2})\oplus (\fh(C^{2})\oplus \fh(C^{(3)})^{\oplus \epsilon}(3)\oplus P(4) \\ &\oplus \fh(C^{(2)})(8g-12)\oplus \fh(C)(8g-11)\oplus \fh(C)(8g-10) \oplus \QQ(8g-8)  
\end{align*}
where $\epsilon=0$ for $g=2$ and $\epsilon =1$ otherwise, and $P$ is a sum of direct factors of motives of smooth projective varieties of dimensions $\leq 8g-15$ (we have $P=0$ for $g=2$, and $P$ can be read off the formula above for $g=3$). This ensures that the term $P(4)$ does not contribute to the Chow groups in the ranges we are considering. All the formulas then follow from this together with the fact that
  \[
\Pic(\Sym^{2}(C))_{\QQ}\simeq \QQ\oplus \Pic(\Jac(C))_\QQ.
\]
which is deduced from the decomposition $\fh(C)=\QQ(0)\oplus \fh^{1}(C)\oplus \QQ(1)$.
\end{proof}  

For $\cN_C(3,d)$ and $\cM_C(3,d)$, the situation is complicated by the Jacobian factor $\fh(\Jac(C))$. Here are some groups which are still reasonably simple to write down.

\begin{cor}\
\begin{enumerate}[label=\emph{(\roman*)}, leftmargin=0.7cm]
\item $\CH^{1}(\cN_C(3,d))_\QQ\simeq \CH^{1}(\cM_C(3,d))_\QQ\simeq \Pic(\Jac(C))_{\QQ}\oplus \QQ$.
\item $\CH^{2}(\cN_C(3,d)_\QQ\simeq \CH^{2}(\cM_C(3,d))_\QQ\simeq \CH^{2}(\Jac(C))_\QQ\oplus \Pic(\Jac(C)\times C)_{\QQ}\oplus \QQ^{2}$.
\end{enumerate}    
\end{cor}
\begin{proof}
  For these two Chow groups, the direct summands of $\fh(\cM_C(3,d))$ other than $\fh(\cN_C(3,d))$ in the decomposition of Theorem \ref{main thm H} do not contribute, so it remains to do the computation for $\cN_C(3,d)$, which follows from the formula in Theorem \ref{thm:main thm N} in a straightforward way. 
\end{proof}

\bibliographystyle{amsplain}
\bibliography{references}

\medskip \medskip

\noindent{Radboud University, IMAPP, PO Box 9010, 6500 GL Nijmegen, The Netherlands} 

\medskip \noindent{\texttt{lie.fu@math.ru.nl, v.hoskins@math.ru.nl, simon.pepinlehalleur@ru.nl}}

\end{document}